\documentclass[a4paper,11pt,twoside]{amsart}
\usepackage[latin1]{inputenc}
\usepackage{amsmath, amsfonts, amssymb,amsthm}
\usepackage[mathscr]{eucal}
\usepackage[pdftex]{graphicx}
\usepackage{color}
\usepackage[T1]{fontenc}
\usepackage[sc]{mathpazo}
\usepackage[all]{xy}
\usepackage{hyperref}
\usepackage{enumerate}

\newtheorem{theorem}{Theorem}[section]
\newtheorem{proposition}[theorem]{Proposition}

\theoremstyle{definition}

\theoremstyle{remark}
\newtheorem{remark}[theorem]{Remark}

\numberwithin{equation}{section}

\title[On stable capillary hypersurfaces with planar boundaries]{On stable capillary hypersurfaces with planar boundaries}

\author{Rabah Souam}
\address{Institut de Math\'{e}matiques de Jussieu-Paris Rive Gauche,   CNRS et Universit\'e de Paris, B\^{a}timent Sophie Germain,  Case 7012, 75205  Paris Cedex 13, France}
\email{rabah.souam@imj-prg.fr}



\begin{document}
\maketitle

\begin{abstract}
We study  stable immersed capillary hypersurfaces $\Sigma$ in domains $\mathcal B$ of $\Bbb R^{n+1}$ bounded by hyperplanes. When $\mathcal B$ is a half-space, we show $\Sigma$ is a spherical cap. When $\mathcal B$ is a domain bounded by 
$k$ hyperplanes $P_1,\ldots,P_k,\, 2\leq k\leq n+1,$ having independent normals, and $\Sigma$ has contact angle $\theta_i$ with  $P_i$ and does not touch the edges of $\mathcal B,$ we prove there exists $\delta>0,$ depending only on $P_1,\dots,P_k,$ so that if 
$\theta_i\in (\frac{\pi}{2}-\delta,  \frac{\pi}{2}+\delta)$ for each $i,$ then $\Sigma$ has to be a piece of a sphere.

\end{abstract}

\section{Introduction}\label{sec:introduction}

A capillary hypersurface in a domain $\mathcal B$ of a Riemannian manifold is a constant mean curvature (CMC)  hypersurface in $\mathcal B$ 
meeting $\partial\mathcal B$ at a constant angle. Capillary hypersurfaces are critical points for an energy functional for variations which preserve the {\it enclosed volume}.  
More precisely, given an angle $\theta\in(0,\pi),$ for a compact hypersurface  $\Sigma$ inside $\mathcal B$ such that $\partial \Sigma \subset \partial \mathcal B$ and $\partial \Sigma$ bounds a compact domain $W$ in
$\partial \mathcal B,$ the energy of $\Sigma$ is  by definition the quantity
\[ E(\Sigma):= |\Sigma| -\cos\theta \,  |W|.
\]
Here and throughout this paper,  we use the notation $|M|$ to denote the volume of a Riemannian manifold $M.$ 
The  energy functional $E$ can be defined for general immersed hypersurfaces, not only those with boundary bounding a compact domain in $\partial\mathcal B.$ Details can be found in \cite{A-S}. 
The stationary hypersurfaces of $E$ for variations preserving the {\it enclosed volume} are precisely the 
CMC hypersurfaces  making a constant angle $\theta$ with  $\partial \mathcal B.$  

Existence and uniqueness problems of capillary hypersurfaces are  interesting on their own right.  In the minimal case,   when $\mathcal B$ is the unit ball in $\Bbb R^3$ and the angle of contact is $\frac{\pi}{2},$ these questions proved to be related to deep problems in Geometric Analysis, see for instance the work of Fraser-Schoen \cite{fraser-schoen}. When $\mathcal B$ is a domain in $\Bbb R^3,$ capillary surfaces correspond to models of incompressible liquids inside a container $\mathcal B$ in the absence of gravity. The free interface of the liquid (locally) minimizes the energy functional $E.$ It is then a natural problem to study the stable ones, that is, those for which the second variation of the energy functional is nonnegative for all volume preserving variations.  Stability issues for capillary hypersurfaces have been  recently actively addressed in different ambient manifolds and domains, see for instance \cite{A-S, choe-koiso, li-xiong, marinov, nunes, souam, wang-xia}. 
One of the achievements in this direction is the proof  by Wang-Xia \cite{wang-xia}  that in a unit ball in a space form, totally umbilical hypersurfaces are the only stable capillary hypersurfaces (Nunes \cite{nunes}  previously solved  the free boundary case  in a unit ball in $\Bbb R^3$).

 In this paper, we are interested in domains $\mathcal B$  in $\Bbb R^{n+1}$  bounded by a finite family of hyperplanes. The first case we consider is when  $\mathcal B$ is a half-space in $\Bbb R^{n+1}.$ In this case, it was shown by  Ainouz-Souam \cite{A-S} that a stable capillary hypersurface with contact angle $\theta< \frac{\pi}{2}$ is a spherical cap provided each of its boundary components is embedded. Choe-Koiso \cite{choe-koiso} obtained the same result when $\theta > \frac{\pi}{2}$ under the assumption that the boundary  is convex and  Marinov \cite{marinov}  treated the case  $n=2$  assuming the boundary is embedded. In our main result (Theorem \ref{thm:half-space}) we remove the boundary embeddedness assumption and characterize the spherical caps as the only stable immersed capillary hypersurfaces supported by a hyperplane  in $\Bbb R^{n+1}, n\geq 2.$ The second case we study is when  $\mathcal B$ is a domain in $\Bbb R^{n+1}$ bounded by a finite number of hyperplanes $P_1,\ldots, P_k$ with linearly independent normals.   Li-Xiong \cite{li-xiong} showed that, in this situation, a stable capillary hypersurface meeting each $P_i$ with a constant angle $\theta_i\in [\pi/2,\pi]$ and not touching the edges of $\mathcal B$ is a piece of a sphere under the assumption that its boundary is embedded for $n=2$ and that each of its boundary components is convex for $n\geq 3.$ The case $k=2,$ that is, when $\mathcal B$ is a wedge, was previously obtained by Choe-Koiso \cite{choe-koiso}. We here prove (Theorem \ref{thm:planar boundaries}) the existence of 
 a positive number $\delta>0$ such that if $\Sigma$ is a stable immersed capillary hypersurface in $\mathcal B$ not touching the edges of $\mathcal B$ and making an angle 
 $\theta_i\in (\frac{\pi}{2} -\delta, \frac{\pi}{2}+\delta)$ with each $P_i,$ then $\Sigma$ is a piece of a sphere. We emphasize that in our results we do not assume the boundary of the hypersurfaces is embedded. 

\section{Preliminaries}

We here recall briefly some basic facts about capillary hypersurfaces and refer  
to \cite{A-S} for more details. 

 Let $n\geq 2$ be an integer and $\mathcal B$ a domain in $\Bbb R^{n+1}$ bounded by a finite number $P_1,\ldots, P_k$  of hyperplanes, $k\geq 1.$ 
 A capillary hypersurface in $\mathcal B$ is a compact and constant mean curvature (CMC)  hypersurface  contained in $\mathcal B$  with boundary on $\partial \mathcal B$ and  meeting each $P_i$  at a constant angle $\theta_i\in(0,\pi).$   
Consider a  capillary hypersurface defined by a smooth immersion $\psi: \Sigma\longrightarrow \mathcal B$ and  let $N$ be a global  unit normal to $\Sigma$ along $\psi$ chosen so that its (constant) mean curvature satisfies $H\geq 0.$ Write $\partial\Sigma=\cup_{i=1}^k \Gamma_i,$ as a union of compact submanifolds with $\psi(\Gamma_i)\subset P_i.$ By an edge of $\mathcal B$ we mean a non empty intersection $P_i\cap P_j$ with  $i\neq j.$ 
We will suppose that the hypersurface does not touch the edges of $\mathcal B$ so that $\psi(\partial\Sigma)$ lies on the smooth parts of $\partial\mathcal B$ where the unit normal is well defined (see however the comments following Theorem \ref{thm:planar boundaries}). 

Capillary  immersions in $\mathcal B$ are critical points of an energy functional for deformations through immersions in 
$\mathcal B,$ mapping  $\partial\Sigma$ into $\partial\mathcal B$ and preserving the {\it enclosed volume}. When $\psi$  is an embedding with each  $\psi(\Gamma_i)$ bounding a domain $W_i\subset P_i,$  the energy writes:
$$E(\psi)= |\psi(\Sigma)| -\sum_{i=1}^k \cos\theta_i |W_i|$$
 and the enclosed volume is the volume of the domain in $\Bbb R^{n+1}$ bounded by $\psi(\Sigma)$ and $\cup_{i=1}^k W_i.$ This energy functional and the notion of enclosed volume can be extended in a suitable way to include general immersions, details can be found in \cite{A-S}. 

We denote by $\bf n_i$ the exterior unit normal to $P_i.$ The angle of contact along $\Gamma_i$ is 
the angle $\theta_i \in (0,\pi)$  between $N$ and $\bf n_i$. We let $\nu_i$ be the exterior unit normal to $\Gamma_i$ in $\Sigma$ and  $\bar\nu_i$ be the unit normal to $\Gamma_i$ in $\partial\mathcal B$ chosen so that 
$\{N,\nu_i\}$ and $\{{\bf n_i }, \bar\nu_i\}$ have the same orientation in $(T\partial\Sigma)^{\perp}.$ The angle between  $\nu_i$ and $\bar\nu_i$ is also equal to $\theta_i.$

 The index form $\mathcal I$ of $\psi$ is the symmetric bilinear form defined on 
$\mathcal C^\infty(\Sigma)$  by
\[ \mathcal I (f,g)=-\int_{\Sigma} f\left(\Delta g+  |\sigma|^2\, g\right) dA+
 \sum_{i=1}^k \int_{\Gamma_i}f\left(\frac{\partial g}{\partial \nu_i}- q_ig\right)\,ds,
\]
where $\Delta$  stands for the Laplacian of the metric induced by $\psi,$  $\sigma$ is the second fundamental form of $\psi$ and  
\[ q_i=  \cot\theta_i \,\sigma(\nu_i,\nu_i). 
\]

 The  capillary immersion $\psi$ as above  is said to be  stable if  the second variation of the energy $E$ is nonnegative for all volume preserving deformations. This condition is  equivalent to  
 $$\mathcal I(f,f)\geq 0\quad {\text  {for all }}\quad f\in \mathcal C^\infty
 (\Sigma)\quad {\text  with} \quad \int_{\Sigma} f\,dA =0,$$
see \cite{A-S}. We note that we can take this as a definition of stability for capillary immersions supported by $\partial \mathcal B$ but not necessarily contained in $\mathcal B.$

Let $D$ denote the standard covariant  differentiation  on $\Bbb R^{n+1}.$ A basic well-known fact  we will use is the following (see for instance \cite{A-S}): 

\begin{proposition}\label{lem:umbilicity}
 Let $\mathcal B$ be as above and  $\psi:\Sigma\longrightarrow \mathcal B$  a capillary immersion into $\mathcal B.$ Then, the unit conormal $\nu_i$ along $\Gamma_i$  is a principal direction of $\psi,$ that is, $D_{\nu_i} N =-\sigma(\nu_i,\nu_i)\,\nu_i.$ 
\end{proposition}

\section{The results}

A key ingredient we will use is the following fact obtained in \cite{A-S}, we include the proof for completeness.

 \begin{proposition}\label{prop:normalintegral}
Let $\psi:\Sigma\to \Bbb R^{n+1}$ be a   smooth immersion in the Euclidean space $\Bbb R^{n+1}$ of a smooth compact orientable $n-$dimensional manifold $\Sigma$ with or without boundary. Let $N:\Sigma \to \Bbb S^n\subset\Bbb R^{n+1}$ be a global unit normal of  $\psi$ and $\nu$ the exterior unit  conormal to $\partial\Sigma$ in $\Sigma$. Then 
\begin{equation}\label{eq:normal}
n\int_{\Sigma} N \, dA=  \int_{\partial\Sigma} \{ \langle \psi,\nu\rangle N - \langle \psi,N\rangle \nu\} \,ds
\end{equation}
where $dA$ and $ds$ denote the volume elements  of $\Sigma$ and $\partial\Sigma,$ respectively.

In particular, if $\Sigma$ has no boundary, then
\begin{equation*}
\int_{\Sigma} N \, dA= {0}.
\end{equation*}

\end{proposition}

\begin{proof}
 Let $a$ be a constant vector field on $\Bbb R^{n+1}.$  Consider the following vector field on $\Sigma $
$$ X=  \langle a, N\rangle \psi^T -  \langle \psi,N\rangle { a}^T, $$
where,  $\psi^{T}={\psi}-\langle \psi,N\rangle N$  (resp. ${ a}^T= {a}-\langle a, N\rangle N$ ) is the projection of $\psi$ (resp. of $ a$) on $T\Sigma$. 
 Using  the following well known formulas that can de thoroughly checked:
$$ {\text {div}} \,\psi^{T} = n +nH\langle \psi,N\rangle,\qquad {\text {div}}\,
 { a}^T= n H \langle a, N\rangle,$$
we compute the divergence of $X$:
\begin{align*} {\text {div}}\, X&= \langle a, N\rangle \,{\text {div}}\, \psi^T + \langle a, D_{\psi^T} N\rangle - \langle \psi, N\rangle \text{div}\, {a}^T - \langle {a}^T, N\rangle -\langle \psi, D_{{a}^T} N\rangle \\
&=   n  \langle a, N\rangle + nH\langle \psi,N\rangle \langle a, N\rangle  +  \langle {a}^T, D_{\psi^T} N\rangle-n H
\langle {a}, N\rangle \langle \psi,N\rangle  -\langle \psi^T, D_{{a}^T} N\rangle \\
&= n\langle a, N\rangle,
\end{align*}
 where we used the symmetry of the Weingarten map, that is,  $\langle {a}^T, D_{\psi^T} N\rangle=\langle \psi^T, D_{{a}^T} N\rangle.$ 
  Integrating on $\Sigma$ and using the divergence theorem we get
 \begin{equation*}
 n\int_{\Sigma}\langle  a, N\rangle \, dA=  \int_{\partial\Sigma}  \{ \langle a, N\rangle\langle \psi,\nu\rangle  - \langle \psi,N\rangle \langle a, \nu \rangle\} ds
  \end{equation*}
 Since this is true for any $a$,   (\ref{eq:normal}) follows.
\end{proof}

We now prove our first result  characterizing the spherical caps as the only stable immersed capillary hypersurfaces in half-spaces in $\Bbb R^{n+1}.$ This was proved under additional assumptions on the boundary and the angle of contact in \cite{A-S}, \cite{choe-koiso} and \cite{marinov}. We note here that it is not necessary to assume the hypersurface in contained in a half-space, we only need  the fact that its meets a hyperplane at a constant angle, so we state the result in this more general form.

\begin{theorem}
\label{thm:half-space}
 Let $\psi:\Sigma\to \Bbb R^{n+1}$ be a capillary immersion  of a compact and connected orientable manifold $\Sigma$ of dimension $n$  in  $\Bbb R^{n+1},\, n\geq 2,$   supported by a hyperplane. If the immersion is stable then 
$\psi(\Sigma)$ 
is a spherical cap.

\end{theorem}

\begin{proof}

We denote, as usual, by  $e_1,\ldots,e_{n+1}$  the vectors of the canonical basis of $\Bbb R^{n+1}$. Without loss of generality, we may suppose the supporting hyperplane is the hyperplane  $\{ x_{n+1}= 0\}$ oriented by the unit normal 
$ -e_{n+1}.$ The starting point is to derive a suitable test function inspired by the work of Barbosa-do Carmo \cite{barbosa-do carmo} proving that round spheres are the only compact and closed stable CMC hypersurfaces in $\Bbb R^{n+1}.$ 

 Integrating the equation 
 $ {\text {div}} \,\psi^{T} = n +nH\langle \psi,N\rangle,$
 we get
 
\begin{equation}\label{eq1}
\int_{\partial \Sigma} \langle \psi,\nu\rangle ds = n\int_{\Sigma} \{1+H\langle \psi,N\rangle\}dA.
\end{equation} 
  
On $\partial\Sigma$, we have  $\, \cos\theta\, N+\sin\theta\, \nu=-e_{n+1}, \,\langle \psi,\nu\rangle=\cos\theta \langle \psi,\bar\nu\rangle$ and  
$\langle \psi,N\rangle=-\sin\theta \langle \psi,\bar\nu\rangle.$  Proposition \ref {prop:normalintegral} gives in this case

  \begin{equation}\label{eq:normalintegral}
 n \cos\theta\int_{\Sigma} N\, dA = -\left( \int_{\partial\Sigma} \langle \psi,\nu\rangle\, \,ds\right)\, e_{n+1}.
  \end{equation}

  From (\ref{eq1}) and (\ref{eq:normalintegral}), we conclude that:
\begin{equation*}
\int_{\Sigma} \{1+H\langle \psi, N\rangle +\cos\theta \langle N,e_{n+1}\rangle\} \,dA =0.
\end{equation*}
So we may use  $\phi:=1+H\langle \psi, N\rangle +\cos\theta \langle N,e_{n+1}\rangle$ as a test  function in the stability inequality.
Set $u=\langle \psi,N\rangle$ and $v=\langle N,e_{n+1}\rangle.$ These functions satisfy the following well-known formulas
\begin{equation}\label{eq3}
\Delta u +|\sigma|^2 u =-nH,
\end{equation}
and 
\begin{equation}\label{eq4}
\Delta v+ |\sigma|^2 v =0.
\end{equation}
Using these equations, we  compute:
\begin{align*}
\Delta \phi&= H(-nH-|\sigma|^2 u)-\cos\theta \, |\sigma|^2 v\\
&=-nH^2-|\sigma|^2(H u+\cos\theta \, v)\\
&=-nH^2-|\sigma|^2(\phi-1).
\end{align*}
Therefore
\begin{equation*}
\Delta\phi+ |\sigma|^2\phi =(|\sigma|^2-nH^2).
\end{equation*}
Moreover:
\begin{equation*}
\frac{\partial u}{\partial \nu}= \langle \nu, N\rangle + \langle \psi, D_{\nu} N\rangle=-\sigma(\nu,\nu)\langle \psi,\nu\rangle\\
\end{equation*}
and 
\begin{equation*} 
\frac{\partial v}{\partial \nu}= \langle D_{\nu}N, e_{n+1}\rangle=-\sigma(\nu,\nu)\langle\nu,e_{n+1}\rangle=\sigma(\nu,\nu)\sin\theta,
\end{equation*}
 Now, taking into account that, along $\partial\Sigma: \,\langle \psi,\nu\rangle=\cos\theta \langle \psi,\bar\nu\rangle$ and  
$\langle \psi,N\rangle=-\sin\theta \langle \psi,\bar\nu\rangle,$    one can check  after direct computations that:
\begin{equation}\label{eq:normalderivative}
\frac{\partial \phi}{\partial \nu}= \cot\theta\, \sigma(\nu,\nu) \,\phi.
\end{equation}
It follows that
\begin{equation}\label{eq7}
\mathcal I(\phi,\phi)= -\int_{\Sigma} (|\sigma|^2-nH^2)\phi\, dA= -\int_{\Sigma} (|\sigma|^2-nH^2)(1+H\langle \psi,N\rangle
 +\cos\theta \langle N,e_{n+1}\rangle) \, dA.
\end{equation}
We now introduce  the following function:
$$ F=H|\psi|^2+2\langle \psi,N\rangle.$$
One has $\Delta|\psi|^2= 2n(1+H\langle\psi,N\rangle)$ and $\Delta  \langle\psi,N\rangle=-|\sigma|^2\langle\psi,N\rangle -nH,$ so that 
\begin{align*} \Delta F &= 2nH(1+H\langle \psi,N\rangle)+2(-|\sigma|^2\langle \psi,N\rangle -nH)\\
&=2(nH^2-|\sigma|^2)\langle\psi,N\rangle
\end{align*}
Integrating we get
\begin{equation*}  2\int_{\Sigma} (nH^2-|\sigma|^2)\langle \psi,N\rangle \,dA = \int_{\Sigma} \Delta F\, dA
=\int_{\partial\Sigma}\frac{\partial F}{\partial\nu}\, ds,
\end{equation*}
that is,
\begin{equation}\label{special function} 
\int_{\Sigma} (nH^2-|\sigma|^2)\langle \psi,N\rangle \,dA = \int_{\partial\Sigma} (H-\sigma(\nu,\nu))\langle\psi,\nu\rangle\, ds.
\end{equation}

Denote by $H_{\partial\Sigma}$ the mean curvature of   the immersion $\psi |_{\partial \Sigma}:\partial \Sigma\longrightarrow \Bbb R^n\times\{0\}$ computed with respect to the unit normal 
$\bar\nu.$ We claim that, on $\partial\Sigma:$ 
\begin{equation}\label{mean curv} H-\sigma(\nu,\nu)=-(n-1)(H+\sin\theta\, H_{\partial\Sigma}).
\end{equation}
Indeed, 
let $i\in\{1,\ldots ,k\}$ and $\{v_1,\ldots,v_{n-1}\}$ be a local orthonormal frame on  $\partial\Sigma.$ Then:
\begin{equation*} 
\sigma(\nu,\nu)=nH-\sum_{j=1}^{n-1} \sigma(v_j,v_j).
\end{equation*}
Now, considering the unit normal $\bar\nu$  to  $\partial\Sigma$ in $\Bbb R^n\times\{0\}$ along $\psi$, as chosen in Sec.\ 2, we have $ N=-\sin\theta\,\bar\nu - \cos\theta \, e_{n+1}.$ We can thus write 
\begin{equation*}\sigma(v_j,v_j)=-\langle \nabla_{v_j} N, v_j\rangle=\sin\theta\, \langle \nabla_{v_j} \bar\nu, v_j\rangle.
\end{equation*}
Therefore the following  relation holds on $\partial\Sigma$
\begin{equation*}\label{eq:sigma}
\sigma(\nu,\nu)=nH+(n-1)\sin\theta\, H_{\partial\Sigma}.
\end{equation*}
This proves (\ref{mean curv}).
Using  (\ref{mean curv}) and the fact that $\langle \psi,\nu\rangle =\cos\theta \langle\psi,\bar\nu\rangle, $ we can write (\ref{special function}) as follows:
\begin{equation}\label{middle intergal}\int_{\Sigma} (nH^2-|\sigma|^2)\langle \psi,N\rangle \,dA= -(n-1)\cos\theta \int_{\partial\Sigma}(H+\sin\theta \, H_{\partial\Sigma})\langle\psi,\bar\nu\rangle \,ds.
\end{equation}

We claim that 
 \begin{equation}\label{claim} \int_{\partial\Sigma}(H+\sin\theta \, H_{\partial\Sigma})\langle\psi,\bar\nu\rangle\,ds=0.
 \end{equation}

To see this, we first apply Minkowski formula  to the immersion $\psi |_{\partial \Sigma}:\partial \Sigma\longrightarrow \Bbb R^n\times\{0\},$ obtaining:
$\int_{\partial\Sigma} H_{\partial\Sigma}\langle\psi,\bar\nu\rangle ds=-|\partial\Sigma|.$ Thus
$$\int_{\partial\Sigma}(H+\sin\theta \, H_{\partial\Sigma})\langle\psi,\bar\nu\rangle\,ds= H\int_{\partial\Sigma}\langle\psi,\bar\nu\rangle -\sin\theta |\partial\Sigma|.$$
Integrating the relation $\Delta \psi =nH N,$ we obtain
\begin{equation*}
\int_{\partial\Sigma} \nu\, ds= nH\int_{\Sigma} N\,dA.
\end{equation*}
Therefore
\begin{equation*}
\int_{\partial\Sigma} \langle \nu,e_{n+1}\rangle \, ds= nH\int_{\Sigma}\langle  N,e_{n+1}\rangle\,dA,
\end{equation*}
that is,
\begin{equation}\label{eq5}
-\sin\theta\,|\partial\Sigma|= nH\int_{\Sigma} \langle  N,e_{n+1}\rangle\,dA.
\end{equation}

Along
 $\partial\Sigma,$ we have  $\langle \psi,\nu\rangle=\cos\theta\langle \psi, \bar\nu\rangle$ and $\langle  \psi, N\rangle =-\sin\theta \langle \psi, \bar\nu\rangle.$   So 
  (\ref{eq:normal})  gives 
\begin{equation}\label{equation 6}
 n\int_{\Sigma} \langle  N,e_{n+1}\rangle\,dA=-\int_{\partial\Sigma} \langle \psi, \bar\nu\rangle\,ds.
\end{equation}
 (\ref{eq5}) and (\ref{equation 6}) show that 
 $$H\int_{\partial\Sigma}\langle\psi,\bar\nu\rangle \,ds -\sin\theta |\partial\Sigma|=0.$$
This proves (\ref{claim}). So (\ref{middle intergal}) gives 
$$H\int_{\Sigma} (nH^2-|\sigma|^2)\langle \psi,N\rangle=0.$$
Going back to (\ref{eq7}), we get
$$\mathcal I(\phi,\phi)= -\int_{\Sigma} (|\sigma|^2-nH^2) (1+\cos\theta\langle N,e_{n+1} \rangle).$$
 As $\theta\in(0,\pi),$ we  have $1+\cos\theta\langle N,e_{n+1} \rangle > 0.$ Also
 $|\sigma|^2-nH^2\geq 0$ with equality only at umbilics. Hence $\mathcal I(\phi,\phi)\leq 0.$ However, by the stability assumption,  $\mathcal I(\phi,\phi)\geq 0.$ We conclude that necessarily $|\sigma|^2-nH^2\equiv 0$ and so 
 $\psi(\Sigma)$ is totally umbilical, that is, a spherical cap.
 \end{proof}


\begin{remark}  More generally we can consider capillary immersions in $\Bbb R^{n+1}$ supported by a hyperplane so that the angle of contact  is locally constant, that is,  this angle is constant along each boundary component.  It is an open problem to decide if spherical caps are the only stable capillary immersions of this type.  Using the argument of Theorem 3.1 in \cite{A-S} (see also \cite{marinov}) one shows this is the case  for surfaces of genus zero in $\Bbb R^3$.
\end{remark}

In the  proof  of Theorem \ref{thm:half-space} we strongly used  in our computations that the origin lies on the supporting hyperplane. This explains the hypothesis 
in our second result dealing with a domain bounded by a finite family of hyperplanes. It was inspired by the work of Li-Xiong \cite{li-xiong}.

\begin{theorem}
\label{thm:planar boundaries}
 Let $\mathcal B$ be a domain in $\Bbb R^{n+1}$ bounded by $k$ hyperplanes  $P_1,\ldots, P_k,\\ 2\leq k\leq n+1,$ having linearly independent normals. Then there exists a constant $\delta>0$ so that  if $\psi:\Sigma\to \mathcal B$  is a stable  compact and connected orientable immersed capillary hypersurface  in $\mathcal B$ not touching the edges of $\mathcal B$ and having a contact   angle $\theta_i\in (\frac{\pi}{2}-\delta
,\frac{\pi}{2}+ \delta)$ with
 $P_i,$ then  $\psi(\Sigma)$ is a part of a sphere. 

\end{theorem}

\begin{proof}
Let  $\bf{n_i}$ denote, as above, the exterior unit normal to   $\mathcal B$ along $P_i.$   The independence of $\bf{n_1},\ldots, \bf{n_k}$   ensures that $\cap_{i=1}^k P_i\neq \emptyset$ and so we may assume, without loss of generality, that the origin lies in $\cap_i^k P_i.$  Following  \cite{li-xiong}, we consider  a linear combination $a:=\sum_{i=1}^n c_i{\bf n_i}$ 
verifying: $\langle a,{\bf{n}}_i \rangle= -\cos\theta_i, \, i=1,\ldots ,k.$ The vector  $a$ exists and is unique by the independence of
 ${\bf{n}}_1,\ldots, {\bf{n}}_k.$ We consider the function 
 $$\phi=  1+H\langle \psi,N\rangle +\langle a,N\rangle.$$
  Proceeding as in the proof of Theorem \ref{thm:half-space}, one shows that 
 $\int_{\Sigma} \phi \,dA=0$  and 
 $$\mathcal I(\phi,\phi) = -\int_{\Sigma} (|\sigma|^2-nH^2) \phi\, dA= -\int_{\Sigma} (|\sigma|^2-nH^2) (1+H\langle \psi, N\rangle + \langle a, N\rangle) \, dA.$$

We claim that
\begin{equation}\label{middle intergal2} \int_{\Sigma} (|\sigma|^2-nH^2) H\langle \psi, N\rangle\,dA=0.
\end{equation}

To prove this, write $\partial\Sigma=\cup_{i=1}^k \Gamma_i,$ with $\psi(\Gamma_i)\subset P_i.$  As in the proof of Theorem \ref{thm:half-space}, considering the function $F=H|\psi|^2+2\langle\psi,N\rangle,$ one shows that
\begin{equation}\label{integral 0} 
\int_{\Sigma} (nH^2-|\sigma|^2)\langle\psi,N\rangle \,dA= -(n-1) \sum_{i=1}^k \cos\theta_i\int_{\Gamma_i}(H+\sin\theta_iH_{\Gamma_i})\langle\psi,\bar\nu_i\rangle\, ds,
\end{equation}
where $H_{\Gamma_i}$ is the mean curvature of the immersion $\psi |_{\Gamma_i}:\Gamma_i \longrightarrow P_i$ with respect to the unit normal $\bar\nu_i.$  As before,  Proposition \ref{prop:normalintegral}  together with  the relation $\langle \psi,\nu_i\rangle=\cos\theta\langle \psi, \bar\nu_i\rangle$ on $\Gamma_i,$ give
\begin{equation}\label{integral 1}
n\int_{\Sigma} N\,dA= \sum_{i=1}^k \left(\int_{\Gamma_i}\langle\psi,\bar\nu_i\rangle\,dA\right){\bf n_i}.
\end{equation}
Integrating the relation $\Delta\psi=nHN,$ we obtain
\begin{equation*} nH\int_{\Sigma} N\,dA= \sum_{i=1}^k \int_{\Gamma_i} \nu_i\,ds=
\sum_{i=1}^k \int_{\Gamma_i}(\cos\theta_i\,\bar\nu_i+\sin\theta_i\,{\bf n_i})\,ds.
\end{equation*}
By Proposition \ref{prop:normalintegral} applied to  the immersion $\psi |_{\Gamma_i:} \Gamma_i\longrightarrow P_i,$ one has $\int_{\Gamma_i}\bar\nu_i\,ds=0.$ So,
\begin{equation}\label{integral 2}
nH\int_{\Sigma} N\,dA= \sum_{i=1}^k  \sin\theta_i |\Gamma_i|{\bf n_i}.
\end{equation}
 From (\ref{integral 1}), (\ref{integral 2}) and the independence of the ${\bf n_i}$'s, we infer that 
 \begin{equation}\label{integral 3} \int_{\Gamma_i}H\langle\psi,\bar\nu_i\rangle\,ds =\sin\theta_i |\Gamma_i|.
 \end{equation}
 
Applying Minkowski formula to the immersion $\psi |_{\Gamma_i}: \Gamma_i\longrightarrow P_i,$ we get
\begin{equation}\label{integral 4}
|\Gamma_i|= -\int_{\Gamma_i} H_{\Gamma_i} \langle \psi,\bar\nu_i\rangle\,ds.
\end{equation}
(\ref{integral 3}) and (\ref{integral 4}) give
$$\int_{\Gamma_i}(H+\sin\theta_iH_{\Gamma_i})\langle\psi,\bar\nu_i\rangle\, ds=0.$$
Together with (\ref{integral 0}) this proves (\ref{middle intergal2}).
Summarizing, we have proved that 
$$\mathcal I(\phi,\phi) = -\int_{\Sigma} (|\sigma|^2-nH^2) (1 + \langle a, N\rangle) \, dA.$$
As in the proof of Theorem \ref{thm:half-space}, we can conclude that $\psi(\Sigma)$ is totally umbilical provided $|a|<1.$ 
Now, recall that $a$ is the unique solution, in the linear space spanned by ${\bf n_1,\dots,n_k},$ to the  system of linear equations $\langle a,{\bf n_i}\rangle=-\cos\theta_i, \,i=1,\ldots,k,$
and depends continuously on the $\theta_i$'s. When $\theta_i=\frac{\pi}{2},$ for $ i=1,\ldots,k,$ the solution to the system  is $a=0.$ Therefore 
there exists $\delta>0$ such that if $\theta_i\in (\frac{\pi}{2}-\delta,\frac{\pi}{2}+\delta), i=1,\ldots,k,$ then $|a|<1$ and $\psi(\Sigma)$ is totally umbilical. Note that $\psi(\Sigma)$ cannot be planar because otherwise it would  not be disjoint from  the edges of $\mathcal B.$ Therefore $\psi(\Sigma)$ is a piece of a sphere.
\end{proof}

Conversely, a part of a sphere as above is stable and even a minimizer of the energy functional $E$  among all embedded hypersurfaces in
$\mathcal B$ enclosing the same volume, see \cite{zia et al.}.

In the free boundary case, that is, when $\theta_i=\frac{\pi}{2},$ for each $i=1,\ldots,k,$  the only possibility for a piece of a sphere to meet orthogonally the hyperplanes $P_i$ is to be centered at a point in the intersection $\cap_{i=1}^k P_i.$ 
The conclusion of  Theorem \ref{thm:planar boundaries} is therefore that a stable  hypersurface with free boundary  has to touch some of the edges of $\mathcal B.$ The case of a wedge, i.e when $k=2,$ was considered by Lopez \cite{lopez}.

It is possible to restate Theorem \ref{thm:planar boundaries} for immersions not necessarily contained in the domain bounded by the hyperplanes and possibly touching its edges. More precisely,  let   $P_1,\dots, P_k, 2\leq k\leq n+1$ be a family of hyperplanes in $\Bbb R^{n+1},$ oriented by given unit normals $\bf{n_1},\dots,\bf{n_k},$ respectively. We consider  capillary immersions  $\psi:\Sigma \longrightarrow \Bbb R^{n+1}$  with $\partial \Sigma$ decomposed into a  union $\cup_{i=1}^k \Gamma_i$ of (not necessarily connected) submanifolds with $\psi(\Gamma_i)\subset P_i$ and 
$\measuredangle (N,{\bf n_i}) =\theta_i$ along $\Gamma_i,\, i=1,\dots,k.$ As above, $N$ denotes the unit normal to the immersion for which the mean curvature verifies  $H\geq 0.$  We note that  this definition excludes, for instance,  the part of a round sphere in a wedge and centered at its edge.  Then  the  conclusion  of Theorem \ref{thm:planar boundaries} is valid for such immersions.

\end{document}